\documentclass[leqno,twoside, 11pt]{amsart}
\usepackage{amssymb, latexsym, esint}
\usepackage{color}
\usepackage{amsthm}
\usepackage{amssymb,amsfonts}
\usepackage[all,arc]{xy}
\usepackage{enumerate}
\usepackage{mathrsfs}
\usepackage{amsmath}
\usepackage{verbatim}
\usepackage[toc,page]{appendix}

\theoremstyle{plain}
\numberwithin{equation}{section} \numberwithin{figure}{section}

\newtheorem{theorem}{Theorem}[section]
\newtheorem{lemma}[theorem]{Lemma}
\newtheorem{proposition}[theorem]{Proposition}
\newtheorem{corollary}[theorem]{Corollary}
\newtheorem{definition}[theorem]{Definition}
  \usepackage{epsfig} 
\usepackage{graphicx}  \usepackage{epstopdf}
\theoremstyle{definition}
\newtheorem{remark}[theorem]{Remark}

\numberwithin{equation}{section}

\begin{document}

\markboth{Pablo Ochoa and Analia Silva}
{$p(x)$-Laplacian}

%

%

\title[$p(\cdot)$-Laplacian]{Effect of non-linear lower order terms in quasilinear equations involving the $p(\cdot)$-Laplacian}

\author{Pablo Ochoa and Analia Silva}

\address[P. Ochoa]{Facultad de Ingenier\'ia. Universidad Nacional de Cuyo-CONICET, Parque Gral. San Mart\'in\\
Mendoza, 5500, Argentina.\\
pablo.ochoa@ingenieria.uncuyo.edu.ar}

\address[A. Silva]{Instituto de Matem\'atica Aplicada San luis (IMASL),
Universidad Nacional de San Luis, CONICET. Ejercito de los Andes
950, D5700HHW, San Luis,
Argentina\\acsilva@unsl.edu.ar\\analiasilva.weebly.com}

\maketitle

\begin{abstract}
In this work, we  study the existence of  $W_0^{1, p(\cdot)}$-solutions to the following boundary value problem involving the $p(\cdot)$-Laplacian operator:

\begin{equation*}
  \left\lbrace
  \begin{array}{l}
       -\Delta_{p(x)}u+|\nabla u|^{q(x)}=\lambda g(x)u^{\eta(x)}+f(x), \quad\textnormal{ in }  \Omega, \\\qquad \,\,\,\,\,\quad  \quad\qquad\quad  u\geq  0, \quad\textnormal{ in }  \Omega\\
    \qquad \,\,\,\,\,\quad  \quad\qquad\quad  u= 0, \,\,\quad  \text{on } \partial\Omega.\\
  \end{array}
  \right.
\end{equation*}under appropriate ranges on the variable exponents. We give assumptions on $f$ and $g$ in terms of the growth exponents $q$ and $\eta$ under which the above problem has a non-negative solution for all $\lambda > 0$.

\end{abstract}


\section{Introduction}

The contribution of the article is to give conditions on $f$ and $g$ to guarantee the existence of weak solutions in  the variable exponent Sobolev space $W_0^{1, p(\cdot)}(\Omega)$ to
boundary value problems with the $p(\cdot)$-Laplacian operator:

\begin{equation}\label{mainproblem}
  \left\lbrace
  \begin{array}{l}
       -\Delta_{p(x)}u+|\nabla u|^{q(x)}=\lambda g(x)u^{\eta(x)}+f(x), \quad\textnormal{ in }  \Omega, \\\qquad \,\,\,\,\,\quad  \quad\qquad\quad  u\geq  0, \quad\textnormal{ in }  \Omega\\
    \qquad \,\,\,\,\quad  \quad\qquad\quad  u= 0, \,\,\quad  \text{on } \partial\Omega.\\
  \end{array}
  \right.
\end{equation}Our results extend the analysis of \cite{APP} and \cite{PP} to the non-standard framework with the difference that we look for solutions in $W_0^{1, p(\cdot)}(\Omega)$ and not only in $W_0^{1, q(\cdot)}(\Omega)$. However, to obtain the desired results we should impose somewhat more regularity on the data.

   We always  assume that $\Omega \subset \mathbb{R}^{N}$ is open, bounded and connected with smooth boundary, $\lambda > 0$ is fixed, and that the exponents satisfy:
\begin{equation}\label{H1}
p, q, \eta \in \mathcal{C}(\overline{\Omega}),\,\, p^{-}:=\min_{\Omega}p(\cdot) >1,\,\, \,   p^{+}:=\max_{\Omega}p(\cdot) <N,\quad  0 \leq  \eta(\cdot) < q^{*}(\cdot)-1.
\end{equation}
The main result of the paper Theorem \ref{existence} states the existence of solutions to \eqref{mainproblem} under the following additional assumptions on the exponents:
\begin{equation}\label{assumpt exp}
\max \left\lbrace p(\cdot)-1, 1 \right\rbrace \leq q(\cdot) < p(\cdot),
\end{equation}and appropriate integrability conditions on $f$ and $g$. Moreover, in Theorem \ref{existencep} we also state existence of $W_0^{1, p(\cdot)}$-solutions for the degenerate case $p(\cdot)\geq 2$  with natural growth in the gradient:
\begin{equation}\label{mainproblemp}
  \left\lbrace
  \begin{array}{l}
       -\Delta_{p(x)}u+|\nabla u|^{p(x)}=\lambda g(x)u^{\eta(x)}+f(x), \quad\textnormal{ in }  \Omega, \\\qquad \,\,\,\,\,\quad  \quad\qquad\quad  u\geq  0, \quad\textnormal{ in }  \Omega\\
    \qquad \,\,\,\,\quad  \quad\qquad\quad  u= 0, \,\,\quad  \text{on } \partial\Omega.\\
  \end{array}
  \right.
\end{equation}under slighter conditions on the non-negative data $f$ and $g$. Indeed, in this case we just require $f \in L^{1}(\Omega)$ and $g \in L^{(q^{*}(\cdot)/\eta(\cdot))'}(\Omega)$, recovering results in the cases of the Laplacian \cite{APP}  and of the p-Laplacian \cite{PP}.
The existence of solutions in $W_0^{1, p(\cdot)}$ to  \eqref{mainproblem} and  \eqref{mainproblemp}  does not follow from the general results from \cite{AAB} and \cite{HS}, which are based on Leray-Lions' Theorem or  Brezis' Theorem for pseudo-monotone operators in separable reflexive spaces. Here, we are not able to use that technique due to the higher range of the exponents (coerciveness is not obtained in general). Thus, our approach is different,  uses truncations and hence  is closer to the arguments in \cite{APP} and \cite{BGO1} (see also  \cite{BG2}, \cite{BM}, \cite{BMP1},\cite{S1} and the reference therein).
 However,  limitations derived from the theory of equations with non-standard growth force to introduce  variations in the proof of the main results.

 Recent systematic study of partial differential equations with variable exponents was motivated by the description of models in electrorheological and thermorheological fluids, image processing, or robotics. As an illustrative example, we discuss the model \cite{CLR} for image restoration. Let us consider an input $I$ that corresponds to shades of gray in a domain  $\Omega \subset \mathbb{R}^{2}$.  We assume that I is made up of the true image $u$ corrupted by the noise and that the noise is additive. Thus, the effect of the noise can be eliminated by smoothing the input, which corresponds to minimizing the energy:
 $$E_1(u)=\int_\Omega |\nabla u(x)|^{2}+ |u(x)-I(x)|^{2}dx.$$Unfortunately, smoothing destroys the small details of the image, so this procedure is not useful. A better approach is the total variation smoothing. Since an edge in the image gives rise to a very large gradient, the level sets around the edge are very distinct, so this method does a good job of preserving edges. Total variation smoothing corresponds to minimizing the energy:
 $$E_2(u)=\int_\Omega |\nabla u(x)|+ |u(x)-I(x)|^{2}dx.$$However, total variation smoothing not only preserves edges, but it also creates edges where there were none in the original image. The suggestion of \cite{CLR} was to ensure total variation smoothing ($p=1$) along edges and Gaussian
smoothing  ($p=2$) in homogeneous regions. Furthermore, it employs anisotropic diffusion ($1 < p < 2$) in
regions which may be piecewise smooth or in which the difference between noise and edges is difficult to
distinguish. Specifically, they proposed to minimize:

 $$E(u)=\int_\Omega \phi(x, \nabla u) + (u-I)^{2}dx$$where:
 $$\phi(x, \xi):= \left\lbrace
  \begin{array}{l}
       \frac{1}{p(x)}|\xi|^{p(x)}, \quad\textnormal{ if  }  |\xi| \leq \beta, \\|\xi|-C(\beta, p(x)), \quad\textnormal{ if }  |\xi| > \beta\\
  \end{array}
  \right.$$where $\beta >0$ and $1 \leq p(x) \leq 2$. According to \cite{CLR}, the main benefit of this model is the manner in which it accommodates the local image information.
Where the gradient is sufficiently large (i.e. likely edges), only  total variation based diffusion will be used. Where
the gradient is close to zero (i.e. homogeneous regions), the model is isotropic. At all other locations,
the filtering is somewhere between Gaussian and total variation based. When  minimizing over $u$ of bounded variations, satisfying given Dirichlet conditions, the associated  flow  is:
$$u_t - \text{div}\left( \phi_r(x, \nabla u)\right) +2(u-I)=0, \text{ in }\Omega \times [0, T],$$with $u(x, 0)=I(x)$, $u$ satisfying the prescribed boundary conditions. Hence, the above model is directly related  to the study of PDE's with the $p(\cdot)$-Laplacian operator:
$$\Delta_{p(x)}u:= \text{div }\left(|\nabla u|^{p(x)-2}\nabla u \right).$$Classical references for existence and regularity of solution for $p(\cdot)$-Laplacian Dirichlet problems are \cite{FZ1}, \cite{FZ} and \cite{FZ2}, among others.

Elliptic equations with first order terms have been largely studied in the literature.  It has been shown in \cite{BOP} that the equation:
$$-\Delta u = \lambda \frac{u}{|x|^{2}} + f(x), \qquad \text{ in  a bounded }\Omega, \, 0 \in \Omega,$$ has in general no solution for a positive $f \in L^{1}(\Omega)$.  Indeed, in \cite[Theorem 2.3]{AP1}, it is proved that a sufficient and necessary condition for existence (for $f \in L^{1}(\Omega)$) is that:
$$|x|^{-2}f \in L^{1}(\Omega).$$In contrast, by adding a quadratic gradient term on the left-hand side, solutions do exist for any $\lambda >0$ and  non-negative $f\in L^{1}(\Omega)$ (see \cite{APP1}). This phenomenon has been studied  in depth in the reference \cite{APP} for problems of the form:
\begin{equation}\label{promb peral}
  \left\lbrace
  \begin{array}{l}
       -\Delta u+|\nabla u|^{q}=\lambda g(x)u+f(x), \quad\textnormal{ in }  \Omega, \\
        \qquad \,\,\,  \quad\qquad  u > 0, \quad  \text{on }\Omega.\\
    \qquad \,\,\,\quad  \qquad u= 0, \quad  \text{on } \partial\Omega.\\
  \end{array}
  \right.
\end{equation}for the range $1 \leq q \leq 2$. Indeed, it is proved that, for $q \in (1, 2]$, and if $g \in L^{1}(\Omega)$ satisfies:
$$g \geq 0, \,g\neq 0, \, \text{ and }\,
C(g, q):=\inf_{\phi \in W_0^{1,
p}(\Omega)}\dfrac{\left(\int_\Omega|\nabla \phi|^{q}
dx\right)^{1/q}}{\int_\Omega g |\phi|dx}>0,$$then  Problem
\eqref{promb peral} admits a distributional solution in $W^{1, q}_0(\Omega)$ for any non-negative $f \in
L^{1}(\Omega)$, and any $\lambda \geq 0$. Under higher integrability assumptions on $f$ and $g$, it is possible to get solutions in $W_0^{1, 2}(\Omega) \cap L^{\infty}(\Omega)$ (see \cite[Theorem 2.4]{APP}). The case of a convex
function of the gradient $\varphi(\nabla u)$ ($q\geq 2$ in
\eqref{promb peral}), $f$ Lipschitz and $\lambda =0$ has been
treated in \cite{Li}. Regarding equations with the $p$-Laplacian
operator, we refer the reader to \cite{PP}.

The paper is organized as follows. In section 2 we introduce the main results of the paper. In section 3 we collect some preliminaries results in the framework of variable exponent spaces. In section 4 we prove Theorem \ref{existence} and, finally, in section 5 we give the proof of Theorem \ref{existencep}.

\section{Main results}

We now give the main results of the paper which  state the existence of solutions to Problems \eqref{mainproblem} and \eqref{mainproblemp}. We start giving the notion of solution that we shall employ in the sequel.
\begin{definition}We say that $u \in W_0^{1, p(\cdot)}(\Omega)$ is a weak solution to Problem \eqref{mainproblem} or  \eqref{mainproblemp}  if $gu^{\eta(\cdot)}\in L^{1}_{loc}(\Omega)$ and:
$$\int_\Omega |\nabla u|^{p(x)-2}\nabla u \cdot \nabla \phi\,dx + \int_\Omega |\nabla u|^{q(x)}\phi\,dx=\lambda\int_\Omega g(x)u^{\eta(x)}\phi\,dx+ \int_\Omega f(x)\phi\,dx$$for all $\phi \in W^{1, p(\cdot)}_0(\Omega)\cap L^{\infty}(\Omega)$.
\end{definition} %

The main contribution of the article is the following existence result for the Dirichlet problem \eqref{mainproblem}.

\begin{theorem}\label{existence} Assume \eqref{H1} and \eqref{assumpt exp}.  Let $f \in L^{q_0}(\Omega)$ be non-negative and $g \in L^{q_1(\cdot)}(\Omega)$, $g \gneqq 0$, where:
\begin{equation}\label{exponents}
q_0:= \left(\frac{Nq^{-}}{N-q^{-}}\right)', \qquad q_1(\cdot):=\left(\frac{q^{*}(\cdot)}{\eta(\cdot)+1}\right)'.
\end{equation}Then there is a weak solution $u \in W^{1, p(\cdot)}_0(\Omega)$ to \eqref{mainproblem}.

\end{theorem}

\begin{remark} Observe that if $g \in L^{q_1(\cdot)}(\Omega) $ then $g\in  L^{\left(\frac{q^{*}(\cdot)}{\eta(\cdot)}\right)'}(\Omega)$. So, for any $\phi \in W_0^{1, p(\cdot)}(\Omega)$, we derive $\phi \in L^{q^{*}(\cdot)}(\Omega)$ and hence $\phi^{\eta(\cdot)} \in L^{q^{*}(\cdot)/\eta(\cdot)}(\Omega)$. By the assumption on $g$ we obtain:
\begin{equation*}
\begin{split}\|g^{\frac{1}{\eta(\cdot)}}\phi\|_{L^{\eta (\cdot)}(\Omega)} & \leq C_0( \eta,q)\|g^{\frac{1}{\eta(\cdot)}}\|_{L^{(q^{*}(\cdot)/\eta(\cdot))'\eta}(\Omega)}\|\phi\|_{L^{q^{*}(\cdot)}(\Omega)}\\& =C_0(g, \eta, q)\|\phi\|_{L^{q^{*}(\cdot)}(\Omega)} \\ & \leq C_0(g, \eta, q) \| \nabla \phi \|_{L^{q(\cdot)}(\Omega)},
\end{split}
\end{equation*}
where we have used Lemma \ref{product}.  As a result:
\begin{equation}\label{cond on g}
C(g,\eta,q):=\inf_{\phi \in W_0^{1,
p(\cdot)}(\Omega)}\dfrac{\|\nabla \phi\|_{L^{q(\cdot)}(\Omega)}}{
\|g^{\frac{1}{\eta(\cdot)}}\phi\|_{L^{\eta(\cdot)}(\Omega)}}>0.
\end{equation}
\end{remark}

\

For the case $p(x)=q(x)$ for all $x \in \Omega$ we have the next result. Regarding the assumption $p(\cdot)\geq 2$, we refer the reader to Remark \ref{remark sing}.
\begin{theorem}\label{existencep} Assume \eqref{H1} and $p(\cdot)\geq 2$. Let $f \in L^{1}(\Omega)$ be non-negative and $g\in  L^{\left(\frac{q^{*}(\cdot)}{\eta(\cdot)}\right)'}(\Omega)$, $g \gneqq 0$. Then there is a weak solution $u \in W^{1, p(\cdot)}_0(\Omega)$ to \eqref{mainproblemp}.\end{theorem}

The constant case is a straightforward consequence of the above results (compare to \cite{PP}).
\begin{corollary}Assume $1 < p < N$ and:
\begin{equation}\label{cond q}
\max \left\lbrace1, p-1, \frac{Np}{N+p}\right\rbrace \leq  q < p.
\end{equation}For non-negative $f \in L^{(q^{*})'}(\Omega)$ and $g \in  L^{\left(q^{*}/p\right)'}(\Omega)$, $g \gneqq 0$, there is a non-negative solution $u \in W_0^{1, p}(\Omega)$ of:
\begin{equation*}
  \left\lbrace
  \begin{array}{l}
       -\Delta_{p}u+|\nabla u|^{q}=\lambda g(x)u^{p-1}+f(x), \quad\text{ in }  \Omega, \\
    \,\,\,\,\quad  \quad\qquad\quad  u= 0, \,\,\quad  \text{ on } \partial\Omega.\\
  \end{array}
  \right.
\end{equation*}
\end{corollary}

Condition $q \geq Np/(N+p)$ in \eqref{cond q} is needed in order to have \eqref{H1} for $\eta=p-1$.

\begin{corollary}Assume $q=p$ and $2 \leq p < N$. Let  $f \in L^{1}(\Omega)$ and $g \in L^{q^{*}/(p-1)}(\Omega)$ be non-negative, $g \gneqq 0$. Then there is a non-negative solution $u \in W_0^{1, p}(\Omega)$ of:
\begin{equation*}
  \left\lbrace
  \begin{array}{l}
       -\Delta_{p}u+|\nabla u|^{p}=\lambda g(x)u^{p-1}+f(x), \quad\text{ in }  \Omega, \\
    \,\,\,\,\quad  \quad\qquad\quad  u= 0, \,\,\quad  \text{ on } \partial\Omega.\\
  \end{array}
  \right.
\end{equation*}
\end{corollary}

\begin{remark}Observe that since $q < N$, we have:
$$(q^{*})' < \frac{N}{q}$$hence our results for the constant case $p=2$ require less regularity of $f$  than in \cite[Theorem 2.4]{APP} to get existence in $W^{1, 2}_0(\Omega)$. However, we impose more regularity on $g$ than the used in \cite{APP}. We believe that the optimal regularity on $g$ in all the above results should be:
$$ g \in L^{(q^{*}(\cdot)/\eta(\cdot))'}(\Omega).$$This remains open and will be treated in a future work.  
\end{remark}

As a concluding remark, we point out that the main results of the paper contribute to the fact that the presence of first-order terms produces regularization effects and permits the existence of solutions. In fact, suppose that for each $f \in L^{1}(\Omega)$ there is a weak (energy) solution $u \in W_0^{1, p(\cdot)}(\Omega)$ to:
$$-\Delta_{p(x)}u = u^{p(x)-1}+f(x) \quad \text{in }\Omega.$$Hence:
$$L^{1}(\Omega) \subset W^{-1, p'(\cdot)}(\Omega)$$which is a contradiction.

\section{Preliminaries}

In this section we introduce basic definitions and preliminary results related to spaces of variable exponent and the related theory of differential equations.

Let:
$$\mathcal{C}_{+}(\overline{\Omega}):=\left\lbrace p \in \mathcal{C}(\overline{\Omega}): p(x)>1 \,\,\text{ for any }\,\, x \in \overline{\Omega}\right\rbrace$$
$$p^{-}:=\min_{\overline{\Omega}}p(\cdot), \quad  p^{+}:=\max_{\overline{\Omega}}p(\cdot).$$We always assume that the variable exponents $p$ are taking in $\mathcal{C}_{+}(\overline{\Omega})$ and satisfy that there is $C>0$ so that:
\begin{equation}\label{assumpt H}
|p(x)-p(y)|\leq C\log|x-y|, \quad \text{for all } x, y \in \Omega.
\end{equation}

We also define the variable exponent Lebesgue space by:
$$L^{p(\cdot)}(\Omega):=\left\lbrace u: \Omega \to \mathbb{R}: u \text{ is measurable and }\int_\Omega |u(x)|^{p(x)}\,dx < \infty\right\rbrace.$$A norm in $L^{p(\cdot)}(\Omega)$ is defined as follows:
$$\|u\|_{L^{p(\cdot)}}:=\inf \left\lbrace \lambda >0: \int_\Omega \Big\vert \dfrac{u(x)}{\lambda}\Big\vert ^{p(x)}\,dx  \leq 1\right\rbrace.$$

We denote by $L^{p'(\cdot)}(\Omega)$ the conjugate space of $L^{p(\cdot)}(\Omega)$, where:
$$\frac{1}{p(\cdot)}+ \frac{1}{p'(\cdot)}=1.$$For the next results see \cite{DHHR}.

\begin{theorem}[H\"{o}lder's inequality]
The space $(L^{p(\cdot)}(\Omega), \|\cdot
\|_{L^{p(\cdot)}(\Omega)})$ is a separable, uniform convex Banach
space. For $u \in L^{p(\cdot)}(\Omega)$ and $v \in
L^{p'(\cdot)}(\Omega)$ there holds:
$$\Big\vert \int_\Omega u v \,dx\Big\vert   \leq \left(\frac{1}{p^{-}}+\frac{1}{(p')^{-}} \right)\|u\|_{L^{p(\cdot)}(\Omega)}\|v\|_{L^{p'(\cdot)}(\Omega)}.$$
\end{theorem}

\begin{proposition}\label{properties modulo} Let:
$$\rho(u)=\int_\Omega |u|^{p(x)}\,dx, \quad u \in L^{p(\cdot)}(\Omega)$$be the convex modular. Then the following assertions hold:
\begin{itemize}
\item[(i)] $\|u\|_{L^{p(\cdot)}(\Omega)} <1$ (resp. $=1, >1$) if and only if $\rho(u)< 1$ (resp. $=1, >1$);
\item[(ii)] $\|u\|_{L^{p(\cdot)}(\Omega)} >1$ implies $\|u\|^{p^{-}}_{L^{p(\cdot)}(\Omega)} \leq \rho(u) \leq \|u\|^{p^{+}}_{L^{p(\cdot)}(\Omega)}$, and $\|u\|_{L^{p(\cdot)}(\Omega)} <1$ implies $\|u\|^{p^{+}}_{L^{p(\cdot)}(\Omega)} \leq \rho(u) \leq \|u\|^{p^{-}}_{L^{p(\cdot)}(\Omega)} $;
\item[(iii)] $\|u\|_{L^{p(\cdot)}(\Omega)}  \to 0$ if and only if $\rho(u)\to 0$, and $\|u\|_{L^{p(\cdot)}(\Omega)}  \to \infty$ if and only if $\rho(u)\to \infty$.
\end{itemize}
\end{proposition}

We now give a useful result in order to work with different variable Lebesgue exponents (see \cite{ER}).
\begin{lemma}\label{product}
Suppose that $p, q \in \mathcal{C}_{+}(\overline{\Omega})$ and that $1 \leq p(\cdot)q(\cdot) \leq + \infty$ for all $x \in \Omega$. Let $f \in L^{q(\cdot)}(\Omega)$, $f$ not identically $0$. Then:
\begin{itemize}
\item[(i)] $\|f\|^{p^{+}}_{L^{p(\cdot)q(\cdot)}(\Omega)} \leq \|f^{p(\cdot)}\|_{L^{q(\cdot)}(\Omega)} \leq \|f\|^{p^{-}}_{L^{p(\cdot)q(\cdot)}(\Omega)} $ if $\|f\|_{L^{p(\cdot)q(\cdot)}(\Omega)} \leq 1$;
\item[(ii)] $\|f\|^{p^{-}}_{L^{p(\cdot)q(\cdot)}(\Omega)} \leq \|f^{p(\cdot)}\|_{L^{q(\cdot)}(\Omega)} \leq \|f\|^{p^{+}}_{L^{p(\cdot)q(\cdot)}(\Omega)} $ if $\|f\|_{L^{p(\cdot)q(\cdot)}(\Omega)} \geq 1$.
\end{itemize}
\end{lemma}

The Sobolev space $W^{1, p(\cdot)}(\Omega)$ is defined as follows ($\nabla u$ denotes the distributional gradient):
$$W^{1, p(\cdot)}(\Omega):=\left\lbrace u \in L^{p(\cdot)}(\Omega): |\nabla u|\in L^{p(\cdot)}(\Omega)\right\rbrace$$equipped with the norm:
$$\|u\|_{W^{1, p(\cdot)}(\Omega)}:=\|u\|_{L^{p(\cdot)}(\Omega)}+\|\nabla u\|_{L^{p(\cdot)}(\Omega)}.$$We denote by $W_0^{1, p(\cdot)}(\Omega)$ the closure of $\mathcal{C}_0^{\infty}(\Omega)$ in $W^{1,
p(\cdot)}(\Omega)$ (one important aspect of  the log-H\"{o}lder condition \eqref{assumpt H} is that $\mathcal{C}_0^{\infty}(\Omega)$ is dense in $W^{1, p(\cdot)}(\Omega)$). The following Sobolev Embedding Theorem for
variable exponent spaces holds.
\begin{theorem}\label{Sob emb}
If $p^{+}< N$, then
$$
0<S(p(\cdot),q(\cdot),\Omega) = \inf_{v\in W^{1,p(\cdot)}_0(\Omega)}
\frac{\|\nabla v\|_{L^{p(\cdot)}(\Omega)}}{\|v\|_{L^{q(\cdot)}(\Omega)}},
$$
for all
$$
1\leq q(\cdot)\le p^*(\cdot) = \frac{Np(\cdot)}{N-p(\cdot)}.
$$
\end{theorem}
\begin{remark}
We need the $q(\cdot)$ exponent to be uniformly subcritical, i.e.
$\inf_\Omega(p^*(\cdot) - q(\cdot)) > 0$ to assure that
$W^{1,p(\cdot)}_0(\Omega)\hookrightarrow L^{q(\cdot)}(\Omega)$ is still
compact.
\end{remark}
We recall that the $p(\cdot)$-Laplace operator is given by:
$$-\Delta_{p(x)}u:= - \text{div }\left(|\nabla u|^{p(x)-2}\nabla u\right).$$Let $X=W_0^{1, p(\cdot)}(\Omega)$. The operator $-\Delta_{p(x)}$ is the weak derivative of the functional $J:X\to \mathbb{R}$:
$$J(u):=\int_\Omega\frac{1}{p(x)}|\nabla u|^{p(x)}\,dx$$in the sense that if $L=J': X \to X^{*}$ then:
$$(L(u), v)=\int_\Omega |\nabla u|^{p(x)-2}\nabla u \nabla v\,dx, \quad u, v \in X.$$
We also recall the following properties.
\begin{theorem}\label{properties}Let $X=W_0^{1, p(\cdot)}(\Omega)$. Then:
\begin{itemize}
\item[(i)] $L: X\to X^{*}$ is continuous, bounded and strictly monotone;
\item[(ii)] $L$ is a mapping of type $(S_{+})$, that is, if $u_n \rightharpoonup u$ in $X$ and:
$$\limsup_{n\to \infty}(L(u_n)-L(u), u_n-u)\leq 0$$then $u_n \to u$ in $X$;
\item[(iii)] $L$ is a homeomorphism.
\end{itemize}

\end{theorem}

We also quote the following useful lemma \cite[Lemma 3.3]{AAB}.
\begin{lemma}\label{useful lemma}Let $1< r(\cdot) < \infty$,  $g \in L^{r(\cdot)}(\Omega)$ and $g_n \in L^{r(\cdot)}(\Omega)$ with $\|g_n\|_{ L^{r(\cdot)}(\Omega)} \leq C$. If $g_n(x)\to g(x)$ a.e. in $\Omega$, then $g_n \rightharpoonup g$ in $ L^{r(\cdot)}(\Omega)$.
\end{lemma}

The next generalization of Lemma 1.17 in \cite{CDG} to the variable exponent setting holds true.
\begin{lemma}\label{util}
Suppose $p(\cdot)\in (1,+\infty)$. Let  $\{u_\epsilon\}_\epsilon$ be a weakly convergent sequence in $L^{p(\cdot)}(\Omega)$ with limit $u$ and let $\{\phi_\epsilon\}_\epsilon$ be a bounded sequence in $L^\infty(\Omega)$ with limit $\phi$ a.e in $\Omega$. Then
$
u_\epsilon\phi_\epsilon\rightharpoonup u\phi \mbox{ weakly in }L^{p(\cdot)}(\Omega).
$
\end{lemma}

\section{Proof of Theorem \ref{existence}}

\subsection{Previous results}\label{previous results}
In this section we give preliminary results in order to prove Theorem \ref{existence} in the next section.

Given a non-negative measurable function $u$, we will consider the usual $k$-truncation functions $T_k$ and $G_k$ defined as:

\begin{equation*}
  T_k(u):=\left\lbrace
  \begin{array}{l}
       u, \quad\textnormal{ if }  |u|\leq k , \\
   k, \quad \,\, \text{if } |u| \geq k.\\
  \end{array}
  \right.
\end{equation*}and:
$$G_k(u):= u- T_k(u).$$Observe that $G_k(u)= 0$ when $u \leq k$.

We start by proving the following technical result.
\begin{lemma}\label{technical lemma powers} Let $0 <p(\cdot)-1 \leq q(\cdot) < p(\cdot)$. Then for any $\varepsilon >0$, there is a constant $C_\varepsilon >0$ so that:
\begin{equation}\label{tech inequ}
s^{q(x)} \leq \varepsilon s^{p(x)} + C_\varepsilon, \quad \text{for all }\, s \geq 0\,\,\text{ and }\, x \in \overline{\Omega}.
\end{equation}
\end{lemma}

\begin{proof}Fix $x \in \overline{\Omega}$. Consider the function $h: (0, \infty) \to \mathbb{R}$:
$$h(s)= \dfrac{\varepsilon s^{p(x)}+ C}{s^{q(x)}}= \varepsilon s^{p(x)-q(x)}+ Cs^{-q(x)},$$where $C \geq 1$ is to be chosen. We have:
$$h'(s)=\varepsilon (p(x)-q(x))s^{p(x)-q(x)-1} -Cq(x)s^{-q(x)-1}.$$The only critical point is:
$$s_0=\left(\dfrac{Cq(x)}{\varepsilon(p(x)-q(x))}\right)^{\frac{1}{p(x)}}.$$Since $h'' >0$ in $(0, \infty)$, $h$ attains its minimum at $s_0$. Observe that:
\begin{equation*}
\begin{split}
h(s_0)& = C^{1-\frac{q(x)}{p(x)}}\varepsilon^{\frac{q(x)}{p(x)}}\left[\left(\frac{q(x)}{p(x)-q(x)}\right)^{1-\frac{q(x)}{p(x)}}+ \left(\frac{q(x)}{p(x)-q(x)}\right)^{-\frac{q(x)}{p(x)}}\right]\\ & \geq C^{1-\frac{q^{+}}{p^{-}}}\varepsilon^{\frac{q^{+}}{p^{-}}}\left[ 1 + \left(\frac{1}{q^{+}}\right)^{\frac{q^{+}}{p^{-}}} \right] \qquad \quad (\text{recall }p(x)-q(x) \leq 1) \\ & \geq 1 \qquad \text{for some appropriate }\, C=C(\varepsilon) >0.
\end{split}
\end{equation*}This proves the lemma. \end{proof}

The following proposition gives the existence of solutions to Problem \eqref{mainproblem} for truncated zero-order terms and bounded data.

\begin{proposition}\label{Propo aux} Let $f, g \in L^{\infty}(\Omega)$ be non-negative and let $k$ be positive. Then there exists a non-negative solution $u_k \in W^{1, p(\cdot)}_0(\Omega)$ to the following equation:
\begin{equation}\label{eq aux}
-\Delta_{p(x)}u + |\nabla u|^{q(x)}= \lambda g(x)(T_k u)^{\eta(x)} + f(x) \quad \text{in }\Omega.
\end{equation}
\end{proposition}
\begin{proof}Let $v_k \in W_0^{1, p(\cdot)}(\Omega)$ be so that:

\begin{equation}\label{eq v_k}
-\Delta_{p(x)}v_k =\lambda g(x) k^{\eta^{+}} +f(x).
\end{equation}Observe that $v_k \in L^{\infty}(\Omega)$ (for instance, by Corollary 3.2 in \cite{HS}).  For each $n$ consider the problem:
\begin{equation}\label{aprox-problem2}
  \left\lbrace
  \begin{array}{l}
       -\Delta_{p(x)}w+G_n(x, w_n, \nabla w_n)=f(x), \quad\textnormal{ in }  \Omega, \\
    \qquad \,\,\,\,  \qquad \quad\quad\qquad\quad  w= 0, \quad  \text{on } \partial\Omega.\\
  \end{array}
  \right.
\end{equation}where for $(x, r, \xi) \in \Omega \times \mathbb{R}\times \mathbb{R}^{N}$:
$$G_n(x, r,\xi)=\begin{cases}
F_n(x, r,\xi)&\mbox{ if }0< r \leq v_k\\
F_n(x, 0,0) &\mbox{ if }r\leq 0\\
F_n(v_k,\nabla v_k)&\mbox{ if }r\geq v_k
\end{cases}$$and:
$$F_n(x, r,\xi)= H_n(x, |\xi| )-\lambda g(x) T_k(r)^{\eta(x)}, \qquad H_n(x, \xi)=\dfrac{|\xi|^{q(x)}}{1 + \frac{1}{n}|\xi|^{q(x)}}.$$By \cite[Theorem 4.1]{AAB}, there is a solution $w_n$ to \eqref{aprox-problem2}. We shall prove that $ w_n\geq 0$ for all $n$. We start by considering truncations of $(-w_n)^+$ for each $M\geq 0$:
$$
(-w_n)^+_M=\begin{cases}
(-w_n)^+ &\mbox{if }(-w_n)^+(x)\leq M\\
M&\mbox{if }(-w_n)^+(x)>M\\
\end{cases}
$$  Also, we define the following auxiliary sets:
$$
\omega_0=\{x\in\Omega:-w_n(x)\geq 0\}
$$
$$
\omega_0^M=\{x\in\Omega:0\leq-w_n(x)\leq M\}
$$
It is clear that:
$$
\begin{cases}
\,\,\,\,\,(-w_n)^+_M=0, \,\,\mbox{ if }x\in\Omega-\omega_0\\
\nabla (-w_n)^+_M=0, \,\,\mbox{ if }x\in\Omega-\omega_0^M.
\end{cases}
$$As a result,  using $(-w_n)_M^{+}$ as a test function in \eqref{aprox-problem2},  we obtain:
\begin{equation}
\begin{split}
0 &\leq  \int_\Omega f(x)(-w_n)_M^{+} \\ & =\int_\Omega |\nabla w_n|^{p(x)-2}\nabla w_n \cdot \nabla (-w_n)_M^{+}+ \int_\Omega G_n(x, w_n, \nabla w_n)(-w_n)_M^{+}\\ &= -\int_{\omega_0^{M}}|\nabla w_n|^{p(x)} + \int_{\omega_0}F_n(x, 0, 0)(-w_n)_{M}\\& =  -\int_{\omega_0^{M}}|\nabla w_n|^{p(x)}.
\end{split}
\end{equation}Thus for all $M\geq 0$:
$$\nabla (-w_n)^{+} = 0\quad a.e. \text{ in }\omega_0^{M}.$$It follows that $\nabla (-w_n)^{+}=0$ a.e. in $\Omega$ and hence, since  $(-w_n)^{+} \in W_0^{1, p(\cdot)}(\Omega)$, $(-w_n)^{+}=0$  a.e. Hence, $w_n \geq 0$ and so $w_n$ solves:
\begin{equation}\label{aprox-problem}
  \left\lbrace
  \begin{array}{l}
       -\Delta_{p(x)}w+H_n(x, \nabla w_n)=\lambda g(x)(T_kw)^{\eta(x)}+  f(x), \quad\textnormal{ in }  \Omega, \\
    \qquad \,\,\,\,  \qquad \quad\quad\qquad\quad  w= 0, \quad  \text{on } \partial\Omega.\\
  \end{array}
  \right.
\end{equation}Observe that by comparison $w_n \leq v_k$ where $v_k$ solves \eqref{eq v_k}, and since $w_n$ is non-negative, we get $\|w_n\|_{L^{\infty}(\Omega)} \leq \|v_k\|_{L^{\infty}(\Omega)}$ for all $n$.

We study now the convergence of $w_n$. Using $w_n$ as a test function in \eqref{aprox-problem}, we derive:
$$\int_\Omega |\nabla w_n|^{p(x)}\,dx+ \int_\Omega H_n (x, |\nabla w_n|)w_n\,dx = \lambda \int_\Omega g(x)(T_k w_n)^{\eta(x)}w_n\,dx + \int_\Omega f(x) w_n\,dx.$$Hence:
$$\int_\Omega |\nabla w_n|^{p(x)}\,dx \leq C(f, g, \Omega, k)$$which implies that there is $u_k \in W_0^{1, p(\cdot)}(\Omega)$ so that $w_n \rightharpoonup u_k$ in $W_0^{1, p(\cdot)}(\Omega)$.  By weak$^{*}$-convergence in $L^{\infty}(\Omega)$ we derive $u_k \leq \|v_k\|_{L^{\infty}(\Omega)}$. We now prove that $w_n \to u_k$ strongly in $W_0^{1, p(\cdot)}(\Omega)$.

Consider $\phi(s) =s \exp\left(\frac{1}{4}s^{2}\right)$, which satisfies:
\begin{equation}\label{property phi}
\phi'(s)-|\phi(s)| \geq \frac{1}{2}.
\end{equation}We use $\phi_n=\phi(w_n-u_k)$ as a test function in \eqref{aprox-problem} and we obtain (we write $\phi_n'= \phi'(w_n-u_k)$):
\begin{equation}\label{eqq1}
\begin{split}&\int_\Omega|\nabla w_n|^{p(x)-2}\nabla w_n \cdot \nabla (w_n-u_k)\phi'_n\,dx + \int_\Omega H_n(\nabla w_n)\phi_n\,dx \\& \qquad \qquad  = \lambda \int_\Omega \left(g(x)[T_kw_n]^{\eta(x)}\phi_n\,dx + f(x) \phi_n\right)\,dx.
\end{split}
\end{equation}Since $\phi_n $ is uniformly  bounded and tends to $0$ as $n\to \infty$, we conclude by Lebesgue Dominated Theorem that the right hand side of \eqref{eqq1} tends to $0$.  Next, by Lemma \ref{technical lemma powers} it follows:
\begin{equation}\label{A}
\begin{split}
&\Big\vert\int_\Omega \dfrac{|\nabla w_n|^{q(x)}}{1+
\frac{1}{n}|\nabla w_n|^{q(x)}} \phi_n\,dx \Big\vert  \leq
\varepsilon \int_\Omega |\nabla w_n|^{p(x)}|\phi_n|\,dx +
C_\varepsilon \int_\Omega |\phi_n|\,dx \\ & \qquad  \leq 2^{p^{+}-1} \left(
\varepsilon \int_\Omega|\nabla w_n - \nabla u_k|^{p(x)}|\phi_n|\,dx
+ \varepsilon\int_\Omega |\nabla u_k|^{p(x)}|\phi_n|\,dx\right) +
C_\varepsilon \int_\Omega|\phi_n|\,dx.
\end{split}
\end{equation}Again by Lebesgue's Theorem, the last two terms converge to $0$ as $n \to \infty$. The first term in \eqref{eqq1} is treated as follows:
\begin{equation}\label{AA}
\begin{split}&\int_\Omega|\nabla w_n|^{p(x)-2}\nabla w_n \cdot \nabla (w_n-u_k)\phi'_n \,dx\\ & \qquad = \int_\Omega(|\nabla w_n|^{p(x)-2}\nabla w_n -|\nabla u_k|^{p(x)-2}\nabla u_k) \cdot \nabla (w_n-u_k)\phi'_n\,dx\\ & \qquad +\int_\Omega |\nabla u_k|^{p(x)-2}\nabla u_k \cdot \nabla (w_n
-u_k)\phi'_n\,dx.
\end{split}
\end{equation}Since $\phi_n'$ is bounded, $|\nabla u_k|^{p(\cdot)-2}\nabla u_k \in L^{p'(\cdot)}(\Omega)$ and $\nabla (w_n-u_k) \rightharpoonup 0$ in $L^{p(\cdot)}(\Omega)$ we derive by Lemma \ref{util} that:
$$\lim_{n \to \infty}\int_\Omega |\nabla u_k|^{p(x)-2}\nabla u_k \cdot \nabla (w_n -u_k)\phi'_n\,dx =0.$$
We will  use the well-known vector inequalities:
$$
(|\xi|^{p(\cdot)-2}\xi-|\eta|^{p(\cdot)-2}\eta)\geq  \left(\frac{1}{2}\right)^{p(\cdot)}|\xi-\eta|^{p(\cdot)} \mbox{ if } p(\cdot)\geq 2.
$$
$$
(|\xi|^{p(\cdot)-2}\xi-|\eta|^{p(\cdot)-2}\eta)\geq (p(\cdot)-1)\frac{|\xi-\eta|^2}{(|\xi|+|\eta|)^{2-p(\cdot)}} \mbox{ if } 1< p(\cdot)< 2.
$$

We introduce the sets:$$\Omega_1= \left\lbrace x \in \Omega: p(x)\geq 2\right\rbrace$$and:$$\Omega_2= \left\lbrace x \in \Omega: p(x)< 2\right\rbrace.$$Now:
$$\int_\Omega |\nabla (w_n-u_k)|^{p(x)}\phi'_n\,dx=\int_{\Omega_1} |\nabla (w_n-u_k)|^{p(x)}\phi'_n\,dx + \int_{\Omega_2} |\nabla (w_n-u_k)|^{p(x)}\phi'_n\,dx.$$We treat first the degenerate case:
\begin{equation}\label{deg case}
\begin{split}
&\int_{\Omega_1} |\nabla (w_n-u_k)|^{p(x)}\phi'_n\,dx\\ & \leq
2^{p^{+}}\int_{\Omega_1}(|\nabla w_n|^{p(x)-2}\nabla w_n -|\nabla
u_k|^{p(x)-2}\nabla u_k) \cdot \nabla (w_n-u_k)\phi'_n\,dx \quad \text{(since $\phi_n' >0$)}\\ & \leq 2^{p^{+}}\int_{\Omega}(|\nabla w_n|^{p(x)-2}\nabla w_n -|\nabla
u_k|^{p(x)-2}\nabla u_k) \cdot \nabla (w_n-u_k)\phi'_n\,dx \\ & \leq
2^{p^{+}}\int_\Omega|\nabla w_n|^{p(x)-2}\nabla w_n \cdot \nabla
(w_n-u_k)\phi'_n\,dx+o(1)   \qquad (\text{by }\eqref{AA})\\ & \leq
2^{2p^{+}-1}\varepsilon \int_\Omega|\nabla w_n - \nabla u_k|^{p(x)}|\phi_n|\,dx
+ o(1)  \qquad (\text{by }\eqref{A}  \text{ and } \eqref{eqq1}).
\end{split}
\end{equation}The uniform boundedness of $w_n$ in $W_0^{1, p(\cdot)}(\Omega)$ and of $|\phi_n|$ in $L^{\infty}(\Omega)$ imply by \eqref{deg case} that:
\begin{equation}\label{deg conclusion}
\limsup_{n \to \infty}\int_{\Omega_1} |\nabla (w_n-u_k)|^{p(x)}\phi'_n\,dx  \leq C\varepsilon.
\end{equation}

Next, writing:
\begin{equation*}
\begin{split}
&\int_{\Omega_2} |\nabla (w_n-u_k)|^{p(x)}\phi'_n\,dx\\&= \int_{\Omega_2} \dfrac{|\nabla (w_n-u_k)|^{p(x)}(\phi'_n)^{\frac{p(x)}{2}}}{(|\nabla w_n|+|\nabla u_k|)^{\frac{(2-p(x))p(x)}{2}}}(\phi'_n)^{1-\frac{p(x)}{2}}(|\nabla w_n|+|\nabla u_k|)^{\frac{(2-p(x))p(x)}{2}}\,dx,
\end{split}
\end{equation*}we obtain by H\"{o}lder's inequality and Lemma \ref{product}, that:
{\small\begin{equation}\label{sing set}\begin{split}
& \int_{\Omega_2} |\nabla (w_n-u_k)|^{p(x)}\phi'_n\,dx \\ & \quad \leq C \Big\|\dfrac{|\nabla (w_n-u_k)|^{p(x)}(\phi'_n)^{\frac{p(x)}{2}}}{(|\nabla w_n|+|\nabla u_k|)^{\frac{(2-p(x))p(x)}{2}}}\Big\|_{L^{2/p(\cdot)}(\Omega_2)}\cdot \Big\|(\phi'_n)^{1-\frac{p(x)}{2}}(|\nabla w_n|+|\nabla u_k|)^{\frac{(2-p(x))p(x)}{2}}\Big\|_{L^{2/(2-p(\cdot))}(\Omega_2)} \\ & \quad  \leq C\max\left\lbrace \left(\int_\Omega \dfrac{|\nabla (w_n-u_k)|^{2}\phi'_n}{(|\nabla w_n|+|\nabla u_k|)^{2-p(x))}} \right)^{2/p^{+}}, \left(\int_\Omega \dfrac{|\nabla (w_n-u_k)|^{2}\phi'_n}{(|\nabla w_n|+|\nabla u_k|)^{2-p(x))}} \right)^{2/p^{-}}\right\rbrace \\ & \quad \leq C\max\left\lbrace \left(\int_{\Omega}(|\nabla w_n|^{p(x)-2}\nabla w_n -|\nabla
u_k|^{p(x)-2}\nabla u_k) \cdot \nabla (w_n-u_k)\phi'_n\,dx \right)^{2/p^{+}}, \left( \cdots \right)^{2/p^{-}}\right\rbrace \,\, \\ & \quad \leq C  \max \left\lbrace \left( \varepsilon\int_\Omega|\nabla w_n - \nabla u_k|^{p(x)}|\phi_n|\,dx\right)^{2/p^{+}},   \left(\varepsilon\int_\Omega|\nabla w_n - \nabla u_k|^{p(x)}|\phi_n|\,dx\right)^{2/p^{-}}\right\rbrace + o(1) \\ & \text{ by } \eqref{AA}, \eqref{eqq1}. \end{split}
\end{equation}}
Using again the boundedness of $w_n$, $u_k$ and $|\phi_n|$ we have by \eqref{sing set} that:
\begin{equation}\label{sing conl}
\limsup_{n \to \infty}\int_{\Omega_2} |\nabla (w_n-u_k)|^{p(x)}\phi'_n\,dx  \leq C \max \left\lbrace\varepsilon^{2/p^{+}}, \varepsilon^{2/p^{-}} \right\rbrace.
\end{equation}Combining \eqref{deg conclusion} and \eqref{sing conl}, observing that  $\phi'_n\geq 1$ and letting $\varepsilon \to 0$, we conclude the strong convergence  of $w_n$ to $u_k$ in $W_0^{1, p(\cdot)}(\Omega)$.

Hence for any $\phi \in W_0^{1, p(\cdot)}(\Omega) \cap L^{\infty}(\Omega)$:
\begin{itemize}
\item $\int_\Omega |\nabla w_n|^{p(x)-2}\nabla w_n \cdot \nabla \phi\,dx \to\int_\Omega |\nabla u_k|^{p(x)-2}\nabla u_k \cdot \nabla \phi\,dx$ since $|\nabla w_n|^{p(\cdot)-2}\nabla w_n$ is bounded in $L^{p'(\cdot)}(\Omega)$ and $|\nabla w_n|^{p(x)-2}\nabla w_n \to  |\nabla u_k|^{p(x)-2}\nabla u_k$ a.e. in $\Omega$, so we may apply Lemma \ref{useful lemma}.
\\
\item $\int_\Omega H_n(x, \nabla w_n) \phi\,dx \to \int_\Omega |\nabla u_k|^{q(x)}\phi\,dx$ again by  Lemma \ref{useful lemma} since $H_n(x, \nabla w_n) \to |\nabla u_k|^{q(x)}$ a.e. in $\Omega$ and $H_n(x, \nabla w_n)$ is bounded in $L^{p(\cdot)/q(\cdot)}(\Omega).$
\\
\item $\int_\Omega \lambda g(x)(T_k(w_n))^{\eta(x)} \phi\,dx\to \int_\Omega \lambda g(x)(T_k(u_k))^{\eta(x)}\phi\,dx$ by Lebesgue's Theorem.
\\
\end{itemize}Therefore, $u_k$ solves \eqref{eq aux}.

\end{proof}

We are now in position to prove Theorem \ref{existence}.

\subsection{Proof of Theorem \ref{existence}} \label{proof thm}

For each $n$, let $g_n=T_n(g)$ and $f_n=T_n(f)$.  By Proposition \ref{Propo aux} there is $u_n \in W_0^{1, p(\cdot)}(\Omega)$, non-negative,  so that:

\begin{equation}\label{mainproblemperturbed}
  \left\lbrace
  \begin{array}{l}
       -\Delta_{p(x)}u_n+|\nabla u_n|^{q(x)}=\lambda g_n(x)(T_nu_n)^{\eta(x)}+f_n(x), \quad\textnormal{ in }  \Omega, \\
    \qquad \,\,\,\,\quad \,\, \quad\qquad\quad \,\, u_n= 0, \quad  \text{on } \partial\Omega.\\
  \end{array}
 \right.
\end{equation}We start assuming that $\| \nabla u_n\|_{L^{q(\cdot)}(\Omega)} \geq 1$ for all $n$.
Taking $T_k(u_n)$ as a test function in \eqref{mainproblemperturbed} we derive:
\begin{equation}\label{est truncadas}
\begin{split}
\int_\Omega |\nabla T_ku_n|^{p(x)}\,dx +\int_\Omega|\nabla
u_n|^{q(x)}T_ku_n\,dx & =\lambda \int_\Omega
g_n(x)(T_nu_n)^{\eta(x)}T_ku_n\,dx + \int_\Omega f_n(x)T_ku_n\,dx \\ &
\leq \lambda k \left( \int_\Omega g_n(x)u_n^{\eta(x)}\,dx\right)
+k\|f_n\|_{L^{1}(\Omega)}.
\end{split}
\end{equation}In the case $\int_\Omega g(x)u_n^{\eta(x)}\,dx\leq 1$ we have:
\begin{equation}\label{aster}
\int_\Omega |\nabla T_ku_n|^{p(x)}\,dx +\int_\Omega|\nabla
u_n|^{q(x)}T_ku_n\,dx  \leq k\left( \lambda +
\|f\|_{L^{1}(\Omega)}\right)
\end{equation}and when $\int_\Omega g(x)u_n^{\eta(x)}\,dx > 1$ by  Young's inequality, Proposition \ref{properties modulo} and \eqref{cond on g} we obtain:
\begin{equation}\label{est truncadas}
\begin{split}
&\int_\Omega |\nabla T_ku_n|^{p(x)}\,dx +\int_\Omega|\nabla
u_n|^{q(x)}T_ku_n\,dx  \leq \lambda k \left( \int_\Omega
g_n(x)u_n^{\eta(x)}\,dx\right) +k\|f_n\|_{L^{1}(\Omega)} \\ &  \leq
\frac{\varepsilon(\lambda k)^{q^-/\eta^{+}}}{q^-/\eta^{+}}\left(
\int_\Omega g_n(x)u_n^{\eta(x)}\,dx\right) ^{q^-/\eta^{+}} +
C(\varepsilon)+k||f||_{L^{1}(\Omega)}\\& \leq
\frac{\varepsilon(\lambda k)^{q^-/\eta^{+}}}{q^-/\eta^{+}}\|
g_n^{1/\eta(\cdot)}u_n\|_{L^{\eta(\cdot)}(\Omega)}^{q^-}+
C(\varepsilon)+k||f||_{L^{1}(\Omega)}\\& \leq
\frac{\varepsilon(\lambda k)^{q^-/\eta^{+}}}{C(g, \eta, q)q^-/\eta^{+}}\| \nabla
u_n\|_{L^{q(\cdot)}(\Omega)}^{q^-}+
C(\varepsilon)+k||f||_{L^{1}(\Omega)}.
\end{split}
\end{equation}Hence:
\begin{equation}\label{key ineqs}
\begin{split}
 &\| \nabla u_n\|_{L^{q(\cdot)}(\Omega)}^{q^-}   \leq\int_\Omega |\nabla u_n|^{q(x)}\,dx \\ & \quad \leq \int_\Omega |\nabla T_k u_n|^{q(x)}\,dx + k \int_{\left\lbrace u_n \geq k\right\rbrace}|\nabla u_n|^{q(x)}\,dx \\ & \quad \leq \int_\Omega |\nabla T_k u_n|^{p(x)}\,dx +  \int_{\left\lbrace u_n \geq k\right\rbrace}|\nabla u_n|^{q(x)}T_ku_n\,dx + |\Omega| \qquad (\text{by Young's inequality}) \\ &  \quad \leq \max\left\lbrace k\left( \lambda + \|f\|_{L^{1}(\Omega)}\right),  \frac{\varepsilon(\lambda k)^{q^-/\eta^{+}}}{C(g, \eta, q)q^-/\eta^{+}}\| \nabla u_n\|_{L^{q(\cdot)}}^{q^-}+ C(\varepsilon)+k||f||_{L^{1}(\Omega)} + |\Omega| \right\rbrace
\end{split}
\end{equation}where we have used \eqref{aster} and  \eqref{est truncadas}. Choosing $\varepsilon$ small, we derive $\|\nabla u_n\|_{L^{q(\cdot)}(\Omega)} \leq C$. Thus up to
a subsequence:

\begin{itemize}
\item $u_n \rightharpoonup u$ in $W_0^{1, q(\cdot)}(\Omega)$;
\item $T_ku_n \rightharpoonup T_ku$ in $W_0^{1, p(\cdot)}(\Omega)$;
\item $u_n \to u$ in $L^{s(\cdot)}(\Omega)$, for $s(\cdot) < q^{*}(\cdot)$.
\end{itemize}If $ \| \nabla u_n\|_{L^{q(\cdot)}(\Omega)} \leq 1$ for a subsequence, we obtain  the same conclusions. Using $\psi_{k-1}(u_n)=T_1(G_{k-1}(u_n))$ as a test function in \eqref{mainproblemperturbed} we derive:
{\small\begin{equation}\label{test psi k}
\int_\Omega|\nabla \psi_{k-1}(u_n)|^{p(x)}\,dx + \int_\Omega
\psi_{k-1}(u_n)|\nabla u_n|^{q(x)}\,dx = \lambda \int_\Omega \left(
g_n(x)(T_nu_n)^{\eta(x)}+f_n(x)\right)\psi_{k-1}(u_n)\,dx.
\end{equation}}The last integral may be divided as:
{\small\begin{equation}\label{divided int}
\lambda \int_{\left\lbrace u_n \geq k \right\rbrace} \left(
g_n(x)(T_nu_n)^{\eta(x)}+f_n(x)\right)\psi_{k-1}(u_n)\,dx+\lambda
\int_{\left\lbrace k-1 \leq u_n \leq k \right\rbrace} \left(
g_n(x)(T_nu_n)^{\eta(x)}+f_n(x)\right)\psi_{k-1}(u_n)\,dx,
\end{equation}}since $\psi_{k-1}(u_n)=0$ if $u_n \leq k-1$. Moreover, since $u_n$ is uniformly bounded in $L^{1}(\Omega)$ we derive by Chebyshev's inequality that:
\begin{equation}\label{pro sets}
\begin{split}
& |\left\lbrace x\in \Omega: k\leq u_n \right\rbrace| \to 0
\end{split}
\end{equation}uniformly in $n$ as $k\to \infty$. By the definition of $\psi_{k-1}$ and  H\"{o}lder's inequality  we have:
\begin{equation}\label{divided int1}
\begin{split}
& \int_{\left\lbrace u_n \geq k \right\rbrace} \left(
g_n(x)(T_nu_n)^{\eta(x)}+f_n(x)\right)\psi_{k-1}(u_n)\,dx+
\int_{\left\lbrace k-1 \leq u_n \leq k \right\rbrace} \left(
g_n(x)(T_nu_n)^{\eta(x)}+f_n(x)\right)\psi_{k-1}(u_n)\,dx \\ & \quad  \leq
 \int_{\left\lbrace u_n \geq k-1 \right\rbrace}
\left(g(x)u_n^{\eta(x)}+f(x)\right)\,dx \\ & \quad \leq\left(
\|g\|_{L^{(q^{*}(\cdot)/\eta(\cdot))'}(\left\lbrace u_n \geq k-1
\right\rbrace)}\|u_n^{\eta(\cdot)}\|_{L^{q^{*}(\cdot)/\eta(\cdot)}(\Omega)}
+ \|f\|_{L^{1}(\left\lbrace u_n \geq k-1
\right\rbrace)}\right) \\ & \quad \leq \max
\left\lbrace \left( \int_{\left\lbrace k-1 \leq u_n \right\rbrace}
g(x)^{[q^{*}(x)/\eta(x)]'}\,dx\right)^{1/\gamma^{-}}, \left(
\int_{\left\lbrace k-1 \leq u_n  \right\rbrace}
g(x)^{[q^{*}(x)/\eta(x)]'}\,dx\right)^{1/\gamma^{+}}
\right\rbrace\|u_n^{\eta(\cdot)}\|_{L^{q^{*}/\eta}(\Omega)} \\ &
\qquad \qquad\qquad \qquad  + \|f\|_{L^{1}(\left\lbrace u_n
\geq k-1 \right\rbrace)}.
\end{split}
\end{equation}

where:
$$\gamma(\cdot)= \dfrac{1}{\left(\dfrac{q'(\cdot)}{\eta(\cdot)} \right)}.$$
Now, by the weak convergence of $u_n$ to $u$ in $W_0^{1, q(\cdot)}(\Omega)$,  there is $C>1$ so that:
$$\int_\Omega u_n^{q^{*}(\cdot)}\,dx \leq C.$$Hence, $u_n^{\eta(\cdot)}$ is bounded in $L^{q^{*}(\cdot)/\eta(\cdot)}(\Omega)$. Moreover, by \eqref{pro sets}:
$$C\max \left\lbrace \left( \int_{\left\lbrace k-1 \leq u_n  \right\rbrace}  g(x)^{[q^{*}(x)/\eta(x)]'}\,dx\right)^{1/\gamma^{-}}, \left( \int_{\left\lbrace k-1 \leq u_n  \right\rbrace}  g(x)^{[q^{*}(x)/\eta(x)]'}\,dx\right)^{1/\gamma^{+}} \right\rbrace + \|f\|_{L^{1}(\left\lbrace u_n \geq k-1 \right\rbrace)}$$  goes to 0 as $k \to \infty$, uniformly in $n$. Thus:
\begin{equation*}
\lambda \int_{\left\lbrace u_n \geq k \right\rbrace} \left(
g_n(x)u_n^{\eta(x)}+f_n(x)\right)\psi_{k-1}(u_n)\,dx+\lambda
\int_{\left\lbrace k-1 \leq u_n \leq k \right\rbrace} \left(
g_n(x)u_n^{\eta(x)}+f_n(x)\right)\psi_{k-1}(u_n)\,dx \to 0
\end{equation*} as  $k \to \infty$ uniformly in $n$. It follows that:
\begin{equation}\label{to vitali}
\lim_{k\to \infty}\int_{\left\lbrace u_n \geq k\right\rbrace}|\nabla
u_n|^{q(x)}\,dx=0, \quad \text{ uniformly in $n$}.
\end{equation}Now we want to prove that for each fix $k$ we have:
$$T_ku_n \to T_ku \quad \text{ strongly in }W_0^{1, q(\cdot)}(\Omega).$$

Take $v_n = \phi(T_k(u_n)-T_k(u))$ as a test function in \eqref{mainproblemperturbed} (where $\phi$ satisfies \eqref{property phi}). We get:
\begin{equation}\label{mn eq 2}
\begin{split}
& \int_\Omega |\nabla u_n|^{p(x)-2}\nabla u_n \cdot \nabla
(T_k(u_n)-T_k(u))\phi'_n\,dx + \int_\Omega |\nabla
u_n|^{q(x)}v_n\,dx \\ & \qquad \qquad  = \int_\Omega f_n(x) v_n\,dx + \int_\Omega \lambda
g_n(x)(T_nu_n)^{\eta(x)} v_n\,dx,
\end{split}
\end{equation}with $\phi'_n =\phi'(T_k(u_n)-T_k(u))$. Firstly, the term:
\begin{equation}\label{f 0}
\int_\Omega f_n(x) v_n\,dx \to 0\mbox{ as }n\to\infty
\end{equation}by Lebesgue's Theorem.  Now we treat the term:
$$\int_\Omega \lambda g_n(x)(T_nu_n)^{\eta(x)} v_n\,dx.$$Since $u_n^{\eta(\cdot)}$ is bounded in $L^{q^{*}(\cdot)/\eta(\cdot)}(\Omega)$, there is $w \in L^{q^{*}(\cdot)/\eta(\cdot)}(\Omega)$ so that:
\begin{equation}\label{weak conv g u}
u_n^{\eta(\cdot)} \rightharpoonup w \quad \text{in }
L^{q^{*}(\cdot)/\eta(\cdot)}(\Omega).
\end{equation}
Since we also have $u_n^{\eta} \to u^{\eta}$ a.e.,  we conclude that $w=u^{\eta(\cdot)}$ by Lemma \ref{useful lemma}. By Egorov's Theorem, for each $\varepsilon$ there is a measurable set $A_\varepsilon$ so that $|A_\varepsilon|< \varepsilon$ and $T_ku_n$ converges to $T_ku$ uniformly in $\Omega \setminus  A_\varepsilon$. Then:
\begin{equation*}
\begin{split}
\lambda \int_\Omega g_n(x) (T_nu_n)^{\eta(x)}v_n\,dx&= \lambda
\int_{\Omega \setminus A_j} g_n
(T_nu_n)^{\eta(x)}\left[\phi(T_ku_n-T_ku)\right]\,dx \\ & + \lambda
\int_{A_j} g_n(x) (T_nu_n)^{\eta(x)}\left[\phi(T_ku_n-T_ku)\right]\,dx  \\
& \leq \lambda o(1)\int_\Omega g(x)u_n^{\eta(x)}\,dx + \lambda \phi(2 k)
\int_{A_j} g(x) u_n^{\eta(x)}\,dx
\end{split}
\end{equation*}When $n \to \infty$, the first term in the last equality tends to $0$ (by \eqref{weak conv g u}, the fact that $g \in L^{(q^{*}(\cdot)/\eta(\cdot))'}(\Omega) $ and the uniform convergence of $T_ku_n$ to $T_ku$) and the last term converges to:
$$\lambda\phi(2k)  \int_{A_j}g(x)u^{\eta(x)}\,dx$$which can be arbitrarily small. Thus:
\begin{equation}\label{g 0}
\lambda \int_\Omega g_n(x) (T_nu_n)^{\eta(x)}v_n\,dx\to 0.
\end{equation}

In \eqref{mn eq 2} we decompose:

$$ \int_\Omega |\nabla u_n|^{p(x)-2}\nabla u_n \cdot \nabla (T_k(u_n)-T_k(u))\phi'_n\,dx $$as the sum:
\begin{equation}\label{eq decoms 1}\begin{split}
&\int_\Omega  |\nabla T_ku_n|^{p(x)-2}\nabla T_ku_n \cdot \nabla
(T_k(u_n)-T_k(u))\phi'_n\,dx \\&+ \int_\Omega  |\nabla
G_ku_n|^{p(x)-2}\nabla G_ku_n \cdot \nabla
(T_k(u_n)-T_k(u))\phi'_n\,dx.
\end{split}\end{equation}Since $G_k(u_n)= 0$ in $\left\lbrace u_n \leq k \right\rbrace$, we have that the last term in \eqref{eq decoms 1} equals:
\begin{equation}\label{int 1}
-\int_\Omega |\nabla G_k(u_n)|^{p(x)-2}\nabla G_k(u_n) \cdot \nabla
T_k(u)\chi_{\left\lbrace u_n\geq k \right\rbrace}\phi'_n\,dx.
\end{equation}

Observe that:

$$ \nabla T_k(u)\chi_{\left\lbrace u_n\geq k \right\rbrace}\phi'_n \to 0$$a.e. in $\Omega$ and by Lebesgue's Theorem, the convergence is in $L^{r(\cdot)}(\Omega)$ for all $r(\cdot)\leq p(\cdot)$. Now we shall prove that there is $C>0$ so that\footnote{Observe that the boundedness of $\nabla G_k(u_n)$ holds automatically when $p=q$ by \eqref{to vitali}, that is the case in \cite{BGO1}.}:

\begin{equation}\label{bbG}
\int_\Omega |\nabla G_k(u_n)|^{p(x)}\,dx \leq C \quad \text{for all $n$}.
\end{equation}Observe that \eqref{bbG} and the boundedness of $T_ku_n$ imply that $u \in W^{1, p(\cdot)}_0(\Omega)$ since:
$$\int_\Omega |\nabla u_n|^{p(x)}dx \leq \int_\Omega |\nabla T_k(u_n)|^{p(x)}dx + \int_\Omega |\nabla G_k (u_n)|^{p(x)}dx \leq C$$for some $C>0$.
Next, to prove \eqref{bbG},  take $ G_k(u_n)$ as a test function in \eqref{mainproblemperturbed} we derive:

\begin{equation*}
\begin{split}
&\int_\Omega |\nabla G_k(u_n)|^{p(x)}\,dx= \int_\Omega |\nabla
u_n|^{p(x)-2}\nabla u_n \cdot \nabla G_k(u_n)\,dx \\ & \qquad \leq
\lambda \int_\Omega g_n(x) T_n(u_n)^{\eta(\cdot)}G_k(u_n) \,dx +
\int_\Omega f_n(x) G_k(u_n)\,dx.
\end{split}
\end{equation*}The uniform boundedness follows by the assumptions on $g$ and $f$ (see the conditions on the exponents \eqref{exponents}) and the fact that $u_n$ is uniformly bounded in $L^{q^{*}(\cdot)}(\Omega)$.  Hence:
$$|\nabla G_k(u_n)|^{p(x)-2}\nabla G_k(u_n)$$is uniformly bounded in $L^{p'(\cdot)}(\Omega)$ for large $n$ and thus \eqref{int 1} is of order $o(1)$.

The first term in \eqref{eq decoms 1} is re-writing as:

\begin{equation}\label{eqq 3}
\begin{split}
& \int_\Omega  |\nabla T_ku_n|^{p(x)-2}\nabla T_ku_n \cdot \nabla
(T_k(u_n)-T_k(u))\phi'_n\,dx  \\ & \qquad  = \int_\Omega
\left(|\nabla T_ku_n|^{p(x)-2}\nabla T_ku_n - |\nabla
T_ku|^{p(x)-2}\nabla T_ku \right) \cdot \nabla
(T_k(u_n)-T_k(u))\phi'_n\,dx \\ &   + \int_\Omega |\nabla
T_ku|^{p(x)-2}\nabla T_ku \cdot \nabla (T_k(u_n)-T_k(u))\phi'_n\,dx
\end{split}
\end{equation}The last term in \eqref{eqq 3} tends to 0 as $ n \to \infty$ by Lemma \ref{util}.  Summarizing, from \eqref{mn eq 2}, \eqref{f 0}, \eqref{g 0},  \eqref{eq decoms 1} and  \eqref{eqq 3}, we obtain:
\begin{equation}\label{used red}
\begin{split}
&0 \leq   \int_\Omega  \left(|\nabla T_ku_n|^{p(x)-2}\nabla T_ku_n -
|\nabla T_ku|^{p(x)-2}\nabla T_ku \right) \cdot \nabla
(T_k(u_n)-T_k(u))\phi'_n\,dx \\& \qquad \qquad = -\int_\Omega
|\nabla u_n|^{q(x)} v_n\,dx + o(1) \\ & \qquad \qquad  =
-\int_{\left\lbrace u_n < k \right\rbrace} |\nabla u_n|^{q(x)}
v_n\,dx -\int_{\left\lbrace u_n \geq  k \right\rbrace} |\nabla
u_n|^{q(x)} v_n\,dx + o(1)  \\ & \qquad \qquad\leq
-\int_{\left\lbrace u_n < k \right\rbrace} |\nabla u_n|^{q(x)}
v_n\,dx + o(1).
 \end{split}
\end{equation}Observe that:
\begin{equation}\label{u menor k}
\int_{\left\lbrace u_n < k \right\rbrace} |\nabla u_n|^{q(x)} v_n\,dx  = \int_{\left\lbrace u_n < k \right\rbrace} |\nabla T_ku_n|^{q(x)} v_n\,dx =  \int_{\Omega} |\nabla T_ku_n|^{q(x)} v_n\,dx.
\end{equation}Since $|\nabla T_ku_n|^{q(x)}$ is bounded in $L^{\frac{p(\cdot)}{q(\cdot)}}(\Omega)$ and $v_n$ is uniformly bounded and converges pointwise to $0$, we derive that $|\nabla T_ku_n|^{q(x)} v_n \rightharpoonup 0$ in  $L^{\frac{p(\cdot)}{q(\cdot)}}(\Omega)$, by Lemma \ref{useful lemma}. Therefore:
$$  \int_\Omega  \left(|\nabla T_ku_n|^{p(x)-2}\nabla T_ku_n - |\nabla T_ku|^{p(x)-2}\nabla T_ku \right) \cdot \nabla (T_k(u_n)-T_k(u))\phi'_n\,dx  =o(1).$$By Theorem \ref{properties}, we derive the strong convergence of $T_ku_n$ to $T_ku$ in $W_0^{1, p(\cdot)}(\Omega)$,
and hence in $W_0^{1, q(\cdot)}(\Omega)$.

Finally, for any $\varphi \in W_0^{1, p(\cdot)}(\Omega) \cap L^{\infty}(\Omega)$, we shall prove that:
\begin{equation}\label{to be convergent}
\int_\Omega |\nabla u_n|^{p(x)-2}\nabla u_n \cdot \nabla \varphi \,dx+ \int_\Omega |\nabla u_n|^{q(x)}\varphi \,dx= \lambda \int_\Omega g_n(x)(T_nu_n)^{\eta(x)}\varphi  \,dx+ \int_\Omega f(x)\varphi\,dx
\end{equation}converges to:
$$\int_\Omega |\nabla u|^{p(x)-2}\nabla u\cdot \nabla \varphi\,dx + \int_\Omega |\nabla u|^{q(x)}\varphi\,dx= \lambda \int_\Omega g(x)u^{\eta(x)}\varphi\,dx + \int_\Omega f(x)\varphi\,dx.$$

For the convergence of the first term we proceed as follows:
\begin{equation*}
\begin{split}
\int_\Omega |\nabla u_n|^{p(x)-2}\nabla u_n \cdot \nabla \varphi\,dx
& =  \int_{\left\lbrace u_n \geq k \right\rbrace}|\nabla
u_n|^{p(x)-2}\nabla u_n \cdot \nabla \varphi\,dx \\ & \qquad +
\int_{\left\lbrace u_n \leq k \right\rbrace}|\nabla
T_ku_n|^{p(x)-2}\nabla T_ku_n \cdot \nabla \varphi\,dx.
\end{split}
\end{equation*}For the last term we have  the facts (consequences of the strong convergence of $T_ku_n$ to $T_ku$):
\begin{enumerate}
\item $|\nabla T_k u_n |^{p(x)-2}\nabla T_ku_n \cdot \nabla \varphi \chi_{\left\lbrace u_n \leq k \right\rbrace} \to |\nabla T_k u|^{p(x)-2}\nabla T_k u \cdot \nabla \varphi \chi_{\left\lbrace u \leq k \right\rbrace}$ a.e. in $\Omega$.
\item $|\nabla T_k u_n|^{p(x)-2}\nabla T_k u_n$ is bounded in $L^{p'(\cdot)}(\Omega)$.
\end{enumerate}Hence, by Lemma \ref{useful lemma}:
$$\lim_{n \to \infty}\int_{\left\lbrace u_n \leq k \right\rbrace}|\nabla
T_ku_n|^{p(x)-2}\nabla T_ku_n \cdot \nabla \varphi\,dx = \int_\Omega |\nabla
T_ku|^{p(x)-2}\nabla T_ku \cdot \nabla \varphi\,dx.$$Thus, by \eqref{to vitali} and the assumption $p(\cdot) -1 \leq q(\cdot)$,  we derive:
\begin{equation}\label{conv t_k u}
\lim_{n \to \infty}\int_\Omega |\nabla u_n|^{p(x)-2}\nabla u_n \cdot \nabla \varphi\,dx= \int_\Omega |\nabla
T_ku|^{p(x)-2}\nabla T_ku \cdot \nabla \varphi\,dx + o(1), \quad \text{as } k \to \infty.
\end{equation}Recalling that $u \in W_0^{1, p(\cdot)}(\Omega)$, it follows that $ |\nabla T_ku|^{p(x)-2}\nabla T_ku $ is bounded in $L^{p'(\cdot)}(\Omega)$, hence making $k \to \infty$ in \eqref{conv t_k u}  and appealing again to Lemma \ref{useful lemma} it follows the desired convergence.

 Next, we deal the second term in \eqref{to be convergent}.  Indeed, we will derive that $|\nabla u_n|^{q(x)}  \to |\nabla u|^{q(x)} $ strongly in $L^{1}(\Omega)$ by appealing to Vitali's Lemma. First, we show that
$|\nabla u_n|^{q(\cdot)}$ is uniformly integrable. Indeed, let
$\varepsilon
>0$. By \eqref{to vitali}, there is $k$ so that:
\begin{equation}\label{ineq 0}
\int_{\left\lbrace u_n \geq k \right\rbrace}|\nabla u_n|^{q(x)}\,dx
< \frac{\varepsilon}{3} \quad \text{ for all }n.
\end{equation}Let now $\delta_0 >0$ be so that for any measurable set $E$ with $|E|< \delta_0$, there holds:
\begin{equation}\label{ineq 1}
\int_E|\nabla T_k u|^{q(x)}\,dx <
\frac{\varepsilon}{3}.
\end{equation}By the strong convergence of $T_ku_n$ to $T_k u$ in $W_0^{1, q(\cdot)}(\Omega)$ we derive that there is $N$ (depending on $\varepsilon$ and $k$) so that $n\geq N$ implies for any $|E| <\delta_0$:
\begin{equation}\label{ineq 2}
\int_E |\nabla T_k u_n|^{q(x)}\,dx < \frac{\varepsilon}{3}+ \int_E
|\nabla T_k u|^{q(x)}\,dx < \frac{2\varepsilon}{3}
\end{equation}in view of \eqref{ineq 1}. Thus, for any $n \geq N$ and any set $|E|< \delta_0$ we have by \eqref{ineq 0} and \eqref{ineq 2} that:
$$\int_E |\nabla  u_n|^{q(x)}\,dx  \leq \int_{\left\lbrace u_n \geq k \right\rbrace \cap E}|\nabla u_n|^{q(x)}\,dx+ \int_E |\nabla T_k u_n|^{q(x)}\,dx < \varepsilon.$$Moreover, for any $i \in \left\lbrace 1, ..., N-1\right\rbrace$, there is $\delta_i >0$ so that for any $|E|<\delta_i$:
$$\int_E |\nabla u_i|^{q(x)}\,dx < \varepsilon, \quad i=1, ..., N-1.$$Therefore, the uniform integrability follows by choosing $\delta = \min \left\lbrace \delta_0, \delta_1, ..., \delta_{N-1}\right\rbrace$. We also observe that, by the strong convergence of truncates, $|\nabla u_n|^{q(x)} \to |\nabla u|^{q(x)}$ a. e. in $\Omega$. Hence, by Vitali's Convergence Theorem, we derive  $|\nabla u_n|^{q(x)}  \to |\nabla u|^{q(x)} $ strongly in $L^{1}(\Omega)$.


Finally, we  treat the statement:
\begin{equation}\label{statement g}
\int_\Omega  g_n (T_nu_n)^{\eta(x)}\varphi \,dx \to  \int_\Omega gu^{\eta(x)}\varphi\,dx \text{ as }n \to \infty.
\end{equation}Write:
\begin{equation*}
\begin{split}
&  \int_\Omega gu^{\eta(x)}\varphi\,dx - \int_\Omega  g_n (T_nu_n)^{\eta(x)}\varphi \,dx   \\ & \qquad = \int_\Omega g(u^{\eta(x)}-u_n^{\eta(x)})\varphi \,dx + \int_\Omega g[u_n^{\eta(x)}-(T_nu_n)^{\eta(x)}]\varphi\,dx  + \int_\Omega (g-g_n)(T_nu_n)^{\eta
(x)}\varphi\,dx.
\end{split}
\end{equation*}Now:
\begin{itemize}
\item The convergence:$$ \int_\Omega g(u^{\eta(x)}-u_n^{\eta(x)})\varphi \,dx \to 0$$ holds by the weak convergence of $u_n^{\eta(x)}$ to $u^{\eta(x)}$  in  $L^{q^{*}(\cdot)/\eta(\cdot)}(\Omega)$ and the assumptions on $g$.
\item Observe:
$$\int_\Omega |u_n^{\eta(x)}-(T_nu_n)^{\eta(x)}|^{q^{*}(x)/\eta(x)}dx \leq \int_{\left\lbrace u_n> n\right\rbrace }|u_n|^{q^{*}(x)}dx \to 0 $$by Theorem \ref{Sob emb}, the fact that $q^{*}(\cdot) \leq p^{*}(\cdot)$ and \eqref{to vitali}. Hence:
$$\int_\Omega g[u_n^{\eta(x)}-(T_nu_n)^{\eta(x)}]\varphi\,dx \to 0$$by H\"{o}lder's inequality.
\item Finally, $$\int_\Omega (g-g_n)(T_nu_n)^{\eta
(x)}\varphi\,dx \to 0$$by H\"{o}lder's inequality, the convergence $g_n \to g$ in $L^{(q^{*}(\cdot)/\eta(\cdot))'}(\Omega)$ and the boundedness of $u_n^{\eta(\cdot)}$ in $L^{q^{*}(\cdot)/\eta(\cdot)}(\Omega)$.
\end{itemize}This proves statement \eqref{statement g} and the proof of the theorem is finished.

\section{Proof of Theorem \ref{existencep}}The proof mainly goes as  for  Theorem \ref{existence} for $p(\cdot) \geq 2$. We point out the differences. Firstly, we choose:
$$\phi=s\exp\left(2^{(4p^{+}-2)}s^{2}\right)$$and \eqref{property phi} is now:
\begin{equation}\label{prop phi new}
\phi'-2^{2p^{+}-1}\phi \geq C>0.
\end{equation}
Next,  \eqref{A} reads as:
\begin{equation}\label{to rm a}
\int_\Omega\dfrac{|\nabla w_n|^{p(x)}}{1 + \frac{1}{n}|\nabla w_n|^{p(x)}}\phi_n dx \leq \int_\Omega |\nabla w_n|^{p(x)}\phi_n dx \leq 2^{p^{+}-1}\int_\Omega |\nabla w_n-\nabla u_k|^{p(x)}\phi_ndx+o(1),
\end{equation}and hence \eqref{deg case} yields:
\begin{equation}\label{new est}
\int_{\Omega} |\nabla (w_n-u_k)|^{p(x)}\phi'_n\,dx \leq 2^{2p^{+}-1} \int_\Omega|\nabla w_n - \nabla u_k|^{p(x)}|\phi_n|\,dx
+ o(1).
\end{equation}The strong converge of $w_n$ to $u_k$ in $W_0^{1, p(\cdot)}(\Omega)$ is obtained appealing to \eqref{prop phi new} and to \eqref{new est}. Moreover, since we are not allowed to use Lemma \ref{useful lemma},  the convergence $\int_\Omega H(x, \nabla w_n)\phi\,dx \to \int_\Omega |\nabla u_k|^{p(x)}\phi\,dx$ may be obtained as\footnote{Observe that this argument is also valid for $q(\cdot) < p(\cdot)$.}:
\begin{equation*}
\begin{split}&\Big\vert \int_\Omega \phi\left(H(x, \nabla w_n)-\dfrac{|\nabla u_k|^{p(x)}}{1 + \frac{1}{n}|\nabla w_n|^{p(x)}} +\dfrac{|\nabla u_k|^{p(x)}}{1 + \frac{1}{n}|\nabla w_n|^{p(x)}}  -|\nabla u_k|^{p(x)} \right)\,dx \Big\vert\\ &  \qquad \quad \leq C \left(  \int_\Omega \vert|\nabla w_n|^{p(x)}-|\nabla u_k|^{p(x)} \vert+  \int_\Omega \left(1-\dfrac{1}{1+\frac{1}{n}|\nabla w_n|^{p(x)}}\right) |\nabla u_k|^{p(x)}\,dx\right)\\
& = o(1),
\end{split}
\end{equation*}where we have used the strong convergence of $w_n$ to $u_k$ in $W^{1, p(\cdot)}_0(\Omega)$ and Lebesgue's Theorem for the last integral. Regarding the proof of Theorem \ref{existence}, we first point out that the boundedness \eqref{bbG} is obtained directly from \eqref{to vitali}. Moreover, the other part  to be changed is \eqref{u menor k}, since we cannot use Lemma \ref{useful lemma}. Now, we write:
\begin{equation*}
\begin{split} &\int_\Omega |\nabla T_k u_n|^{p(x)}v_n\,dx \\ & \quad  = \int_\Omega  |\nabla T_k u_n|^{p(x)-2}\nabla T_ku_n \cdot \nabla T_ku_n \,v_n\,dx  + \int_\Omega |\nabla T_k u|^{p(x)-2}\nabla T_k u \cdot \nabla (T_ku_n-T_ku)v_n\,dx  \\ &\quad -\int_\Omega |\nabla T_k u|^{p(x)-2}\nabla T_k u \cdot \nabla (T_ku_n-T_ku)v_n\,dx  + \int_\Omega |\nabla T_k u_n|^{p(x)-2}\nabla T_k  u_n \cdot \nabla T_k u \,v_n\,dx\\ & \quad-\int_\Omega |\nabla T_k u_n|^{p(x)-2}\nabla T_k  u_n \cdot \nabla T_k u \,v_n\,dx \\ & \quad=  \int_\Omega \left( |\nabla T_k u_n|^{p(x)-2}\nabla T_ku_n -|\nabla T_k u|^{p(x)-2}\nabla T_k u \right)\cdot \nabla (T_ku_n-T_ku)v_n\,dx + o(1),
\end{split}
\end{equation*}where the terms:
$$\int_\Omega |\nabla T_k u|^{p(x)-2}\nabla T_k u \cdot \nabla (T_ku_n-T_ku)v_n\,dx $$and:$$ \int_\Omega |\nabla T_k u_n|^{p(x)-2}\nabla T_k  u_n \cdot \nabla T_k u \,v_n\,dx $$converge to $0$ by Lemma \ref{util}. Hence, by \eqref{used red}, it follows:
\begin{equation*}
\begin{split}&\int_\Omega  \left(|\nabla T_ku_n|^{p(x)-2}\nabla T_ku_n -
|\nabla T_ku|^{p(x)-2}\nabla T_ku \right) \cdot \nabla
(T_k(u_n)-T_k(u))\phi'_n\,dx \\ & \quad\leq \int_\Omega \left( |\nabla T_k u_n|^{p(x)-2}\nabla T_ku_n -|\nabla T_k u|^{p(x)-2}\nabla T_k u \right)\cdot \nabla (T_ku_n-T_ku)|v_n|\,dx + o(1).
\end{split}
\end{equation*}Appealing to \eqref{prop phi new}, we derive the strong convergence of $\nabla w_n$ to $\nabla u_k$ in $L^{p(\cdot)}(\Omega)$. The rest of the proof is the same as for Theorem \ref{existence}.

\begin{remark}\label{remark sing} Regarding the extension of Theorem \ref{existencep} to all values of $p(x)$, we point out that in the singular framework, the absence of $\varepsilon$ in \eqref{to rm a} brings difficulties in order to deal with  inequality \eqref{sing set} and hence to obtain the key control \eqref{sing conl}.
\end{remark}


\section*{Acknowledgements}

A.S. is supported by 
PICT 2017-0704, by Universidad Nacional de San Luis under grants
PROIPRO 03-2418 and PROICO 03-1916.  P. O. is supported by Proyecto Bienal  B080 Tipo 1 (Res. 4142/2019-R).

\end{document}